\documentclass[12pt, twoside, leqno]{article}
\usepackage{amsmath,amsthm}
\usepackage{amssymb,latexsym}
\usepackage{enumerate}
\usepackage[all]{xy}

\newtheorem{thm}{Theorem}[section]

\theoremstyle{definition}
\newtheorem{defin}[thm]{Definition}
\newtheorem{rem}[thm]{Remark}
\newtheorem{exa}[thm]{Example}

\numberwithin{equation}{section}

\def\N{\mathbb N}
\def\Z{\mathbb Z}
\def\Q{\mathbb Q}

\def\P{\mathbb P}

\begin{document}

\baselineskip=17pt

%%%%%%%%%%%%%%%%

\title{On $\theta$-congruent numbers, rational squares in arithmetic progressions, 
concordant forms and elliptic curves}

\author{Erich Selder\\
Fachhochschule Frankfurt\\
D - 60318 Frankfurt am Main, Germany\\
E-mail: e$\underbar{\phantom{e}}$selder@fb2.fh-frankfurt.de
\and 
Karlheinz Spindler\\
Hochschule RheinMain\\
D - 65197 Wiesbaden, Germany\\
E-mail: karlheinz.spindler@hs-rm.de}

\date{}

\maketitle

%% Classification and key words; the 2010 classification is used:

\renewcommand{\thefootnote}{}



\renewcommand{\thefootnote}{\arabic{footnote}}
\setcounter{footnote}{0}

%%%%%%%%

\begin{abstract}
The correspondence between right triangles with rational sides, triplets of rational squares in arithmetic 
succession and integral solutions of certain quadratic forms is well known. We show how this 
correspondence can be extended to the generalized notions of rational $\theta$-triangles, rational 
squares occurring in arithmetic progressions and concordant forms. In our approach we establish 
one-to-one mappings to rational points on certain elliptic curves and examine in detail the role of solutions 
of the $\theta$-congruent number problem and the concordant form problem associated with 
nontrivial torsion points on the corresponding elliptic curves. This approach allows us to combine  
and extend some disjoint results obtained by a number of authors, to clarify some statements 
in the literature and to answer some hitherto open questions. 
\end{abstract}

\section{Introduction}

The following definition dates back to Euler ([E]; see also [O2]).

\begin{defin}
Two quadratic forms $X^2+mY^2$ and $X^2+nY^2$ (where $m,n\in{\Z}\setminus\{ 0\}$ with 
$m\not= n$) are called concordant if the system 
$$X^2+mY^2=Z^2, \quad X^2+nY^2=W^2$$ 
admits a nontrivial solution $(X,Y,Z,W)\in{\Z}^4$, where nontriviality means that 
$Y\not= 0$.  This is equivalent to saying that there are solutions $(X,Y,Z,W)$ $\in 
{\P}_3({\Q})$ other than $(1,0,\pm 1,\pm 1)$. Thus $(X,Y,Z,W)$ is a trivial solution if 
and only if it is a solution for any pair $(m,n)$.
\end{defin}

It is easily verified that every system of concordant forms is equivalent to one in which the coefficients 
$m$ and $n$ have different signs; hence we may always assume that $m<0$ and $n>0$. After factoring out the 
greatest common divisor of the coefficients, this leads us to consider quadratic forms 
$X^2-pkY^2$ and $X^2+qkY^2$ where $k,p,q\in{\N}$ with $p$ and $q$ coprime. Concordant 
forms in this form tie up nicely with rational squares occurring in arithmetic progressions. In fact, 
if $\alpha^2<\beta^2<\gamma^2$ are squares of rational numbers which occur in an arithmetic 
progression of (maximally chosen) step size $k$, then there are coprime numbers $p,q\in{\N}$ such 
that $\alpha^2=\beta^2-pk$ and $\gamma^2=\beta^2+qk$, which means exactly that the forms 
$X^2-pkY^2$ and $X^2+qkY^2$ are concordant.  To have a precise terminology available, let us 
give a formal definition. \bigskip 

\begin{defin}
A triplet $(p,q,k)\in{\N}\times{\N}\times{\N}$ where $p,q$ are coprime is called a solution of the 
concordant form problem if and only if the quadratic forms $X^2-pkY^2$ and $X^2+qkY^2$ are 
concordant; i.e., if and only if there is an arithmetic progression of (maximal) step size $k$ 
containing three rational squares, where the lowest and the highest are separated from the 
intermediate one by $p$ times resp. $q$ times the step size. 
\end{defin}

Obviously, a triplet $(p,q,k)$ is a solution of the concordant form problem if and only if 
$(p,q,a^2k)$ is for any $a\in{\N}$ (since the factor $a$ can be subsumed into $Y$); hence it is 
sufficient to study solutions $(p,q,k)$ where $k$ is squarefree. We note that arithmetic 
progressions of squares have been studied not just over the rationals, but over number 
fields (cf. [Co2], [GS], [GX], [X]). While in these approaches the goal was to find (maximal) 
uninterrupted arithmetic progressions of squares in the given base field, we focus on rational 
squares which occur in arithmetic progressions, but not necessarily in immediate succession. 
We now relate the concordant form problem to a different problem which is cast in geometric 
rather than arithmetic language. The concept of congruent numbers (see [Ko], [T]) has been 
extended to that of a $t$-congruent number (cf. [TY]) and even more generally to that of a 
$\theta$-congruent number (cf. [F1], [Ka], [Y1], [Y2], [Y3]); for an overview see [TY]. 
Even though the concept of a $t$-congruent number is more natural from a geometric point 
of view (arising from the search for triangles with rational sides and rational area), the more 
general concept of a $\theta$-congruent number is more relevant for the purposes of this paper. 

\begin{defin}
Given an angle $\theta\in (0,\pi)$ whose cosine is a rational number, a number $k\in{\N}$ 
is called $\theta$-congruent if there is a triangle with rational sides which has $\theta$ as 
an angle and $k\sqrt{s^2-r^2}$ as its area, where $r\in{\Z}$ and $s\in{\N}$ are the unique 
coprime numbers such that $\cos (\theta ) = r/s$. Somewhat more precisely we call a 
triplet $(r,s,k)\in{\Z}\times{\N}\times{\N}$ a solution of the generalized congruent number 
problem if $k$ is $\theta$-congruent where $\cos\theta =r/s$ in lowest terms. 
\end{defin}

Scaling the sides of a triangle by a factor $a$ changes the area by the factor $a^2$; hence 
a natural number is $\theta$-congruent for some angle $\theta$ if and only if its squarefree 
part is. In other words, $(r,s,k)$ is a solution of the generalized congruent number problem if 
and only if $(r,s,a^2k)$ is for any number $a\in{\N}$. Clearly, $(\pi /2)$-congruence is just ordinary 
congruence of numbers; the only other angles $\theta$ for which $\theta$-congruence has 
been studied somewhat systematically are $\theta=\pi /3$ and $\theta =2\pi /3$; see [Y1], 
[Y2], [Y3] and [JSDP]. \bigskip 

The two problems are closely related, as is shown by the following theorem (whose elementary 
proof we omit). \bigskip 

\begin{thm} 
Let $\frak{N}$ be the set of solutions of the generalized congruent number problem, and let 
$\frak{F}$ be the set of solutions of the concordant form problem. Then mutually inverse bijections 
$f:\frak{N}\rightarrow\frak{F}$ and $g:\frak{F}\rightarrow\frak{N}$ are given by 
$$f(r,s,k) := \begin{cases} 
     (s-r,s+r,k) &\hbox{if $r\not\equiv s$ mod $2$},\\ 
     \bigl( (s-r)/2, (s+r)/2, 2k\bigr) &\hbox{if $r,s$ are both odd} \end{cases} $$ 
and 
$$g(p,q,k) := \begin{cases} \bigl( (q+p)/2, (q-p)/2, k\bigr) &\hbox{if $p,q$ are both odd},\\ 
  \bigl( q+p,q-p,k/2\bigr) &\hbox{if $p\not\equiv q$ mod $2$.} \end{cases}$$ 
\end{thm}

\section{Connections to elliptic curves}

In the introduction we exhibited a correspondence between $\theta$-congruent numbers, 
rational squares in arithmetic progressions and concordant forms written in a certain way. Now 
we establish a 1-1 correspondence between the intersection of quadrics given by concordant 
forms and elliptic curves in standard form. 

\begin{thm} Let $m\not= n$ be nonzero integers. We denote by $Q(m,n)$ the set of all 
$(X_0,X_1,X_2,X_3)\in\P_3(\Q)$ such that 
$$X_0^2+mX_1^2=X_2^2\quad\hbox{and}\quad X_0^2+nX_1^2=X_3^2.$$ 
Also, we denote by $E(m,n)$ the set of all $(T,X,Y)\in{\P}_2(\Q)$ such that 
$$Y^2T = X(X+mT)(X+nT)$$ 
which, in affine notation, is just the elliptic curve $y^2=x(x+m)(x+n)$. Then mutually 
inverse isomorphisms $\varphi :Q(m,n)\rightarrow E(m,n)$ and $\psi :E(m,n)\rightarrow 
Q(m,n)$ are given by 
$$\varphi :\left[\begin{matrix} X_0\\ X_1\\ X_2\\ X_3\end{matrix}\right]\mapsto 
  \left[\begin{matrix} nX_2-mX_3+(m\! -\! n)X_0\\ mn(X_3-X_2)\\ 
  mn(m\! -\! n)X_1 \end{matrix}\right]$$ 
and 
$$\psi :\left[\begin{matrix} T\\ X\\ Y\end{matrix}\right]\mapsto 
  \left[\begin{matrix}
       -(X\! +\! mT)\bigl( Y^2-m(X\! +\! nT)^2\bigr) \\ 
       2Y(X\! +\! nT)(X\! +\! mT) \\ 
       -(X\! +\! mT)\bigl( Y^2+m(X\! +\! nT)^2\bigr) \\ 
       -(X\! +\! nT)\bigl( Y^2+n(X\! +\! mT)^2\bigr) \end{matrix}\right]\, .$$ 
\end{thm}

Note that $\varphi$ needs to be redefined at $(1,0,1,1)$ whereas $\psi$ needs to be 
redefined at $(1,-m,0)$, $(1,-n,0)$ and $(0,0,1)$ to obtain well-defined regular maps.
We omit the proof that $\varphi$ and $\psi$ are in fact well-defined and have the 
desired properties (which is obtained by straightforward calculations) and merely 
remark that this isomorphism is an instance of a general correspondence between elliptic 
curves and intersections of quadrics; see [Ca], p. 36, and [Co1], pp. 123-125.) We note 
that the isomorphism $\psi$ maps the point at infinity and the $2$-torsion points 
$(-m,0)$, $(-n,0)$ and $(0,0)$ of $E(m,n)$ exactly to the trivial solutions $(1,0,\pm 
1,\pm 1)$ of $Q(m,n)$. Consequently, all other rational points of $E(m,n)$ correspond 
to nontrivial points on $Q(m,n)$. Thus the following is true (cf. [F1], Prop. 3; 
[Ka], Thm 1). \smallskip 

\begin{thm} Let $m\not= n$ be nonzero integers. Then the quadratic forms $X^2+mY^2$ and 
$X^2+nY^2$ are concordant if and only if $E(m,n)$ possesses elements of (finite or infinite) 
order greater than two. 
\end{thm} 

The fact that our mappings $\varphi$ and $\psi$ are true isomorphisms make them more suitable 
than other correspondences studied in the literature. Let us explain this statement in some 
detail.  

The book [Ko] deals with the classical congruent number problem. If $n$ is a natural number and 
$(X,Y,Z)$ are rational sides of a right triangle with area $n$ (where $X<Y<Z$), then the 
assignment $(X,Y,Z)\mapsto(Z^{2}/2,(X^{2}-Y^{2})Z/8)=(x,y)$ (see [Ko], Ch. I, \S 2, Prop. 2; Ch. I, 
\S 9, Prop. 19) maps the right triangles to rational points of the elliptic curve $E(-n,n)$ given 
by the affine equation $y^{2}=x^{3}-n^{2}x$. The assignment $(X_0, X_1, X_2, X_3)\mapsto( X_2/X_1, 
X_0/X_1, X_3/X_1)$ associates with any point $(X_0,X_1,X_2,X_3)$ on the intersection of the quadrics 
$X_0^2-nX_1^2=X_2^2$ and $X_0^2+nX_1^2=X_3^2$ with $X_i\geq 0$ and $X_1>0$ such a triangle. The 
composite function $\tau$ which is computed to be 
$$(X_0, X_1, X_2, X_3)\ \mapsto\ \left( \frac{X_0^2}{X_1^2}, -\frac{X_0X_2X_3}{X_1^3}\right)$$ 
obviously extends to a regular morphism $Q(-n,n)\rightarrow E(-n,n)$. The trivial elements of $Q(-n,n)$ 
are mapped to the point at infinity, i. e., to the neutral element of the elliptic curve $E(-n,n)$. 
This morphism, however, is not an isomorphism, but is a mapping of degree 4 whose image is exactly 
$2E(-n,n)$, i.e., the set of all doubled points of $E(-n,n)$. If we denote by $\mathbb{D}$ the doubling 
$P\mapsto2P$ on the elliptic curve $E(-n,n)$, then an easy computation shows that the diagram 
$$\xymatrix { & & Q(-n,n) \ar[rr]^{\hbox{$\tau$}} \ar[rd]_{\hbox{$\varphi$}} & & E(-n,n) \\ 
  & & & E(-n,n) \ar[ru]_{\hbox{$\mathbb{D}$}} }$$ 
commutes in which $\varphi$ is the mapping introduced at the beginning of this section. In other words, 
for any point $S$ in $Q(-n,n)$ we have $\tau(S)=2\varphi(S)$. Since the only torsion points on $E(-n,n)$ 
are the points of order 2, these points are mapped to the neutral element of $E(-n,n)$. This diagram  
explains why, for example, the points $(41^2/7^2, 29520/7^3)$ on $E(-31,31)$ and $(25/4,75/8)$ on 
$E(-5,5)$ (cf. [Ko], p. 7) are not in the image of $\tau$. However, they correspond to solutions
of the associated concordant form problem via the mapping $\varphi$. 

Note that the defect of the mapping $\tau$ of not being an isomorphism does not affect any of the assertions  
in [Ko]. In particular, the statement that $n$ is a congruent number if and only if $E(-n,n)$ contains 
rational points of infinite order is true, since the only points of finite order (the 2-torsion points) 
are mapped to the neutral element of $E(-n,n)$ via the mapping $\tau$, and they correspond to the trivial 
solutions of the concordant form problem and the degenerate right triangle. However, the defect of $\tau$ 
would make itself felt if one considered not only rational points on the elliptic curve in question, but also 
solutions over number fields. 

The use of a correspondence which is not an isomorphism runs into problems when torsion points of 
order greater than 2 occur, which is the case in the general concordant form problem considered in [O1]. 
The mapping $\sigma:Q(m,n)\rightarrow E(m,n)$ used in [O1] is given by 
$$(X_0, X_1, X_2, X_3)\ \mapsto\ \left(\frac{X_0^2}{X_1^2},\, \frac{X_0X_2X_3}{X_1^3}\right)$$
and hence is the same as the one considered by Koblitz up to the sign in the second component. It turns out 
that $\sigma$ is again the composition of the isomorphism $\varphi$ with an algebraically defined 
endomorphism of $E(m,n)$, namely the negative $-\mathbb{D}$ of the doubling mapping $\mathbb{D}$. Again, 
this mapping is of degree 4 and has $2E(m,n)$ as its image. But in this more general situa\-tion, the curve  
$E(m,n)$ may have points of order 4, and these points are mapped to a 2-torsion point in the image of $\sigma$. 
Now such 2-torsion points correspond to nontrivial solutions of the associated concordant
form problem. So for any pair of numbers $(m,n)$ for which the torsion subgroup of the elliptic curve $E(m,n)$ 
is isomorphic to $\mathbb{Z}_{2}\times\mathbb{Z}_{4}$ there exist nontrivial solutions of the concordant form 
problem in the sense of Definition 1.1 for the pair $(m,n)$ even if the rank of $E(m,n)$ is zero. Hence the statement 
in [O1] that if $E(m,n)$ has rank zero then nontrivial solutions exist if and only if the torsion group is ${\Z}_2\times 
{\Z}_8$ or ${\Z}_2\times{\Z}_6$ ([O1], p. 101) needs qualification. For example, if $(m,n)=(-1,3)$ then $(X,Y,Z,W) = 
(\pm 1, \pm 1, 0, \pm 2)$ are solutions of the equations $X^2+mY^2=Z^2$ and $X^2+nY^2=W^2$ corresponding to the 4-torsion 
points $(3,\pm 6)$ and $(-1,\pm 2)$ of the elliptic curve $y^2=x(x-1)(x+3)$; note that $E(-1,3)$ has rank 
zero (see [Y1]). These solutions are not covered by Main Corollary 1 (pp. 104/105) and Corollary 2 (p. 107) in [O1] 
where only solutions are considered for which all components are nonzero. 

We note in passing that the curve $y^2=x(x-1)(x+3)$ also provides a counterexample to Proposition 5.4 in [TY], 
which does not hold for $n=1$. In fact, $n=1$ is a $\pi/3$-congruent number (realized by the equilateral 
triangle with all sides equal to 2), which, in fact, corresponds to the 4-torsion points $(3, \pm 6)$ and 
$(-1,\pm 2)$ of the elliptic curve $E(-1,3)$. Hence the additional assumption $n>1$ is indispensible to make 
Proposition 5.4 in [TY] true. 

In the recent paper [Im], again only solutions with nonzero components are considered, as becomes clear from 
Definition 2 in this paper (which should presumably state that a solution $(X,Y,Z,W)$ is considered nontrivial 
only if $XYZW\not= 0$). Thus, again, solutions associated with 4-torsion points on the corresponding curve 
are lost. For example, for $m=1$ and $n=k^2$ where $k\in\{ 2,3,4,5,6,8,9,13\}$ solutions exist (namely  
$(X,Y,Z,W)=(0,1,1,k)$) which are not covered in [Im]. The condition that the rank of $E(1,k^2)$ be zero is satisfied 
for the quoted values of $k$ (whereas the rank is one for 
$k\in\{ 7,10,11,12\}$), as we verified with the SAGE software package. 

While a thorough analysis of the mapping $\sigma$ used in [O1] together with the two-descent on the elliptic 
curve $E(m,n)$ may reveal all the interesting phenomena concerning solutions to the concordant form problem 
from the study of rational points on $E(m,n)$ via $\sigma$, the approach via the isomorphism $\varphi$ seems 
to be much more direct and natural. Also, the correspondence between solutions $(r,s,k)$ of the congruent 
number problem and rational points of order $>2$ on the elliptic curve 
$$y^2 = x\bigl( x-(s\! -\! r)k\bigr)\bigl( x+(s\! +\! r)k\bigr)$$ 
becomes, via the correspondence with rational squares in arithmetic progressions, more lucid 
than in [F1], [Ka] and [Y1]. It also becomes clear from our calculations that two rational points 
of order $>2$ on the curve yield the same triangle if and only if they differ by a 2-torsion element 
in the Mordell-Weil group of the curve. Moreover, our calculations clarify Theorem 1 and Proposition 4 
in [F1]. More precisely, we will exactly determine those numbers $n$ occurring as $\theta$-congruent numbers 
corresponding to torsion elements of the associated elliptic curve, which was left open in [F1] and was 
clarified in the paper [F2], of which we were made aware only after finishing our paper.  
\bigskip 

\section{Nontrivial torsion solutions} 
When we speak of the torsion subgroup of an elliptic curve $E$ over $\Q$ we always mean the torsion 
subgroup of the Mordell-Weil group $E_{\mathbb{Q}}$ of rational points on $E$. A deep theorem by 
Mazur (see [M]) states that the torsion subgroup of any elliptic curve over $\Q$ must be one of the 
groups ${\Z}_m$ where $1\leq m\leq 10$ or $m=12$ or else of the groups ${\Z}_2\times{\Z}_{2n}$ where 
$1\leq n\leq 4$. Since $E(m,n)$ has three points of order $2$, namely $(0,0)$, $(-m,0)$ and $(-n,0)$, 
only the last four possibilities can occur for the curves studied here. We want to determine the exact 
conditions on $m$ and $n$ which determine the type of the torsion group. To do so, we compute all 
nontrivial torsion elements. This was essentially already done in [O1], and only a few additional 
calculations (omitted here) are needed to arrive at the following complete characterization of all 
torsion elements. (Also see [F2].) 

\begin{thm} We consider the elliptic curve $E(m,n)$ over $\Q$ where $m=-pk$ and $n=qk$ 
such that $p,q\in{\N}$ are coprime and $k\in{\N}$ is squarefree. 
\begin{enumerate}[\upshape (i)]
\item There are points of order $4$ if and only if $-m$ and $n-m$ are squares, say 
  $-m=u^2$ and $n=v^2-u^2$. In this case, the $4$-torsion points are exactly the 
  four points 
  $$\bigl( u^2-uv,\, \pm v(u^2-uv)\bigr)\quad\hbox{and}\quad \bigl( u^2+uv,\, 
  \pm v(u^2+uv)\bigr)\, .$$ 
\item There are points of order $8$ if and only if there are numbers $\xi,\eta\in
  {\N}$ such that $\xi^2+\eta^2$ is a square, say $\xi^2+\eta^2=\zeta^2$, and the 
  equations $m=-\xi^4$ and $n=\eta^4-\xi^4=\zeta^2(\eta^2-\xi^2)$ hold. In this case, 
  the $8$-torsion points are exactly the eight points 
  \begin{eqnarray*} 
  &\bigl(\ \xi\zeta (\xi\! +\! \eta )(\zeta\! + \!\eta), 
        \ \pm \xi\eta\zeta (\xi\! +\! \eta)(\zeta\! +\!\xi)(\zeta\! +\! \eta)\ \bigr), \\
  &\bigl(\ \xi\zeta (\xi\! +\! \eta )(\zeta\! - \!\eta), 
        \ \pm \xi\eta\zeta (\xi\! +\!\eta)(\zeta\! -\!\xi)(\zeta\! -\!\eta)\ \bigr), \\ 
  &\bigl(\ \xi\zeta (\xi\! -\!\eta )(\zeta\! + \!\eta), 
        \ \pm \xi\eta\zeta (\xi\! -\!\eta)(\zeta\! -\!\xi)(\zeta\! +\!\eta)\ \bigr), \\ 
  &\bigl(\ \xi\zeta (\xi\! -\!\eta )(\zeta\! - \!\eta), 
        \ \pm \xi\eta\zeta (\xi\! -\!\eta)(\zeta\! -\!\eta)(\zeta\! +\!\xi )\ \bigr).
  \end{eqnarray*} 
\item There are points of order $3$ (or, equivalently, points of order $6$) if and 
  only if there are coprime integers $a,b\not= 0$ with $a+2b\not=0$, $b+2a\not= 0$ 
  and $a\pm b\not= 0$ such that $m=a^3(a+2b)$ and $n=b^3(b+2a)$. In this case 
  the points of order $3$ are the two points 
  $$(a^2b^2,\, \pm a^2b^2(a+b)^2\bigr),$$ 
  and the points of order $6$ are the six points 
  \begin{eqnarray*} 
  &\bigl(\ -a^2b(b\! +\! 2a),\ \pm a^2b(b\! +\! 2a)(a^2\! -\! b^2)\ \bigr),\\ 
  &\bigl(\ -ab^2(a\! +\! 2b),\ \pm ab^2(a\! +\! 2b)(a^2\! -\! b^2)\ \bigr),\\ 
  &\bigl( ab(a\! +\! 2b)(b\! +\! 2a),\ \pm ab(a\! +\! 2b)(b\! +\! 2a)(a\! +\! b)^2\ \bigr).
  \end{eqnarray*}  
\item In all other cases for $m$ and $n$ the only torsion points are 
  the trivial points $(0,0)$, $(-m,0)$ and $(-n,0)$. 
\end{enumerate} 
\end{thm} 

\begin{rem} If $k=d^2\ell$ is not squarefree, then the elliptic curves $E(-p\ell ,q\ell)$ 
and $E(-pk,qk)$ are isomorphic (as algebraic groups) via the isomorphism $(x,y)\mapsto 
(d^2x,d^3y)$. So the torsion subgroups of the corresponding Mordell-Weil groups of 
rational points are also isomorphic. As a special case, we note that the curve 
$E(-pk,qk)$ contains points of order $4$ if and only if $-m=pk$ and $n-m=qk+pk$ are 
squares (irrespectively of whether or not $k$ is squarefree). 
\end{rem} 

We are now ready to give a complete classification of all concordant forms and 
$\theta$-congruent triangles which correspond to torsion solutions of the associated 
elliptic curve. This classification is based on the following theorem. (Cf. [F2].) 

\begin{thm} We consider the elliptic curve $E(m,n)$ where $m=-pk$ and $n=qk$ such that 
$p,q\in{\N}$ are coprime and $k\in{\N}$ is squarefree. Let $T$ be the torsion subgroup of 
$E(m,n)$. 
\begin{enumerate}[\upshape (i)]
\item If $T\cong {\Z}_2\times{\Z}_4$ or $T\cong {\Z}_2\times{\Z}_8$ then $k=1$.
\item If $T\cong{\Z}_2\times{\Z}_6$ then $k=1$ or $k=3$. 
\end{enumerate}
$\bigl($The examples in the next section will show that each of the possible cases occurs 
for an infinite number of elliptic curves $E(m,n)$.$\bigr)$  
\end{thm}

\begin{proof} Let us first consider the case that that $T\cong{\Z}_2\times{\Z}_4$ or 
$T\cong{\Z}\times{\Z}_8$. Assume $k\not=1$; then there is a prime divisor $t$ of $k$. 
The number $t$ then divides both $-m=pk$ and $n-m=(p+q)k$ which are squares according to 
Thm 3.1(i). Since $k$ is squarefree this implies that $t$ divides both $p$ and 
$p+q$, which is impossible because $p$ and $q$ are coprime by assumption. Thus the assumption 
$k\not= 1$ is wrong, and we must have $k=1$ in this case. \par\smallskip 
Let us now consider the case that $T\cong{\Z}_2\times{\Z}_6$. By Thm 3.1(iii) there are coprime 
numbers $a,b\in{\Z}\setminus\{ 0\}$ with $a+2b\not=0$, $b+2a\not=0$ and $a\pm b\not= 0$ such 
that $m=a^4+2a^3b$ and $n=2ab^3+b^4$. Let $t$ be a prime divisor of $k$. Then $t$ divides 
both $-pk=m=a^3(a+2b)$ and $qk=n=b^3(2a+b)$. It is obvious that if the prime $t$ were a 
divisor of $a$ then it would also be a divisor of $b$, and vice versa; this, however, is impossible 
because $a$ and $b$ are coprime. Thus $t$ divides neither $a$ nor $b$, hence divides both $a+2b$ 
and $2a+b$, hence divides $2(a+2b) = (2a+b)+3b$, hence divides $3b$ and consequently must be $3$. 
We have shown that $3$ is the only possible prime divisor of $k$. This implies that 
$k=1$ or $k=3$. 
\end{proof} 

The implication of this theorem for concordant forms and $\theta$-congruent numbers will be 
elucidated in the next section. 

\section{Interpretation and conclusions}

Theorem 3.3 may be interpreted in terms of the concordant form problem and of the generalized 
congruent number problem.\\

\begin{thm} Let $m=-pk$ and $n=qk$ such that $p,q\in{\N}$ are coprime and $k\in{\N}$ is squarefree.
\begin{enumerate}[\upshape (i)]
\item If the quadratic forms $X^2+mY^2$ and $X^2+nY^2$ are concordant due to $4$- or $8$-torsion, 
  then $k=1$. 
\item If the quadratic forms $X^2+mY^2$ and $X^2+nY^2$ are concordant due to $3$- or $6$-torsion then 
  $k=1$ or $k=3$.  
\end{enumerate}
\end{thm}
\begin{proof} This is an immediate consequence of Thm 3.3. 
\end{proof} 

\begin{thm} Let $\theta=\arccos (r/s)$ where $r\in{\Z}$ and $s\in{\N}$ are coprime numbers such that 
$|r|<s$. Moreover, let $k\in{\N}$ be a squarefree number.  
\begin{enumerate}[\upshape (i)] 
\item If $k$ is odd and is a $\theta$-congruent number due to $4$- or $8$-torsion then $k=1$. 
\item If $k$ is even and is a $\theta$-congruent number due to $4$- or $8$-torsion then $k=2$. 
\item If $k$ is odd and is a $\theta$-congruent number due to $3$- or $6$-torsion then $k=1$ or $k=3$.
\item If $k$ is even and is a $\theta$-congruent number due to $3$- or $6$-torsion then $k=2$ or $k=6$.
\end{enumerate} 
\end{thm}

\begin{proof} We remember that the associated elliptic curve is given by $E(m,n)$ where
$m=-(s-r)k$ and $n=(s+r)k$ if $r\not\equiv s$ mod 2 (first case) and where $m=-(s-r)k/2$ and 
$n=(s+r)k/2$ if $r,s$ are both odd (second case). 
In the first case, the coefficients satisfy the hypotheses of Thm 3.3 and we see that in the 
situation of (i) we have $k=1$ and for (iii) we have $k=1$ or $k=3$. The same holds in the second case 
when $k$ is odd. It remains to consider the second case with $k=2\ell$ being even (and $\ell$ being odd 
since $k$ is assumed to be squarefree). In this situation the elliptic curve $E(m',n')$ associated with 
the triplet $(r,s,\ell )$ is given by $m'=-(s-r)\ell /2$ and $n'=(s+r)\ell /2$, and these coefficients 
satisfy the hypotheses of Thm 3.3. So we can conclude that $\ell =1$ (and hence $k=2$) in the situation 
(ii) and $\ell =1$ or $\ell =3$ (and hence $k=2$ or $k=6$) in the situation (iv).
\end{proof} 

Any of the above situations occurs for infinitely many elliptic curves. The characterization 
of the curves with a prescribed torsion subgroup, together with the arguments in the proof of Thm 
3.3, gives rise to several series of examples which will illustrate this fact. 

\begin{exa} {\bf Torsion solutions of order 4.} 
Let $u,v$ be any coprime numbers with $u<v$ and let $m=-u^{2}$ and $n=v^{2}-u^{2}$. Then the torsion 
subgroup $T\subseteq E_{\mathbb{Q}}(m,n)$ contains $\mathbb{Z}_{2}\times\mathbb{Z}_{4}$. Hence any point 
of order 4 defines a solution of the concordant form problem given by the triplet $(p,q,k)=(-m,n,1)$. 
If $m$ and $n$ are both odd (which is equivalent to saying that $v$ is even) then this point defines
a solution to the generalized congruent number problem given by the triplet $(r,s,k)=((n-m)/2,\, (n+m)/2,\, 
1)$ as well. If $m$ and $n$ have different parities, then the elliptic curve $E(4m,4n)$ also has a point
of order 4, which defines a solution to the generalized congruent number problem given by the triplet 
$(n-m,\, n+m,\, 2)$. Explicit instances are given by $(u,v)=(1,2)$, which corresponds to $(m,n)=(-1,3)$ 
and represents the situation that $v$ is even, and by $(u,v)=(1,3)$, which corresponds to $(m,n)=(-1,8)$ and 
represents the situation that $v$ is odd). 
\end{exa}  

\begin{exa} {\bf Torsion solutions of order 8.} Let $(\xi,\eta,\zeta)$ be any primitive Pythagorean triplet 
so that $\xi^{2}+\eta^{2}=\zeta^{2}$ and $\xi<\eta$ (note that $\xi$ and $\eta$ then automatically have 
unequal parities!) and let $m=-\xi^{4}$ and $n=\eta^{4}-\xi^{4}$. Then the torsion subgroup $T\subseteq 
E_{\mathbb{Q}}(m,n)$ is isomorphic to $\mathbb{Z}_{2}\times\mathbb{Z}_{8}$. So any point of order 8 defines a 
solution of the concordant form problem given by the triplet $(p,q,k)=(-m,n,1)$. If $m$ and $n$ are both odd 
(which means $\xi$ is odd and $\eta$ is even), then this point also defines a solution of the generalized 
congruent number problem given by $(r,s,k)=((n-m)/2,\, (n+m)/2,\, 1)$. If $m$ is even (and \it a fortiori \rm 
$n$ is odd), then the elliptic curve $E(4m,4n)$ also has points of order $8$, each of which defines a solution
of the generalized congruent number problem given by the triplet $(n-m,\, n+m,\, 2)$. Explicit instances are 
given by $(\xi,\eta,\zeta)=(3,4,5)$, which corresponds to $(m,n)=(-81,175)$ and represents the situation 
of a primitive Pythagorean triplet with $\xi$ being odd), and by by $(\xi,\eta,\zeta)=(8,15,17)$, which 
corresponds to $(m,n)=(-4096,\, 46\, 529)$ and represents the situation of a primitive Pythagorean triplet with 
$\xi$ being even). 
\end{exa} 

\begin{exa} {\bf Torsion solutions of order 3 or 6.} Let $a,b$ be any coprime numbers with $a<0$, $b>0$, 
$a+2b>0$, $2a+b>0$ and $a+b\neq 0$, and let $m=a^3(a+2b)$ and $n=b^3(2a+b)$. Then the torsion subgroup 
$T\subseteq E_{\mathbb{Q}}(m,n)$ is isomorphic to $\mathbb{Z}_{2}\times\mathbb{Z}_{6}$. From the proof of Thm 
3.3 we know that the only possible common divisors of $m$ and $n$ are 1 or 3.
\begin{itemize}
\item First case: $\gcd (a+2b,2a+b)=1$. In this case any point of order 3 or 6 defines a solution of
the concordant form problem given by the triplet $(p,q,k)=(-m,n,1)$. If $m$ and $n$ are both odd then 
this point also defines a solution of the generalized congruent number problem given by $(r,s,k)=
((n-m)/2,\, (n+m)/2,\, 1)$. If $m\not\equiv n$ mod 2 then the elliptic curve $E(4m,4n)$ also has
points of order 3 and 6, which define solutions of the generalized congruent number problem given by the 
triplet $(n-m,\, n+m,\, 2)$.
\item Second case: $\gcd (a+2b,2a+b)=3$. (Note that this situation occurs when $a\equiv b$ mod 3.) Let 
$p=-m/3$ and $q=n/3$. Then any point of order 3 or 6 defines a solution of the concordant form problem given 
by the triplet $(p,q,k)=(-m/3,\, n/3,\, 3)$. If $m$ and $n$ are both odd, then this point also defines a 
solution of the generalized congruent number problem given by $(r,s,k)=((n-m)/6,\,(n+m)/6,\, 3)$. If 
$m\not\equiv n$ mod 2 then the elliptic curve $E(4m,4n)$ also has points of order 3 and 6, which define 
solutions to the generalized congruent number problem given by the triplet $((n-m)/3,\, (n+m)/3,\, 6)$.
\end{itemize} 
Explicit instances are given as follows: 
\begin{itemize}
\item the example $(a,b)=(-1,3)$ corresponds to $(m,n)=(-5,27)$ and represents the situation that $a,b$ are 
  odd coprime numbers with different signs satisfying the congruence condition $a\not\equiv b$ mod $3$ and the 
  inequalities $a+2b>0$ and $2a+b>0$; 
\item the example $(a,b)=(-2,5)$ corresponds to $(m,n)=(-64,\, 125)$ and represents the situation 
  that $a,b$ are coprime numbers with different signs and different parities satisfying the congruence condition 
  $a\not\equiv b$ mod $3$ and the inequalities $a+2b>0$ and $2a+b>0$; 
\item the example $(a,b)=(-5,\, 13)$ corresponds to $(m,n)=(-875\cdot 3,\, 2197\cdot 3)$ and represents the 
  situation that $a,b$ are odd coprime numbers with different signs satisfying the congruence condition 
  $a\equiv b$ mod $3$ and the inequalities $a+2b>0$ and $2a+b>0$; 
\item the example $(a,b)=(-2,7)$ corresponds to $(m,n)=(-32\cdot 3,\, 343\cdot 3)$ and represents the 
 situation that $a,b$ are coprime numbers with different signs and different parities satisfying the 
congruence condition $a\equiv b$ mod $3$ and the inequalities $a+2b>0$ and $2a+b>0$.
\end{itemize} 
\end{exa} 

\begin{rem} From the above considerations it is clear that each of the possible situations occurs for 
an infinite number of cases. Furthermore, with these examples we answer a question left open in [F1] 
(Remark 1 after Proposition 4). Namely, in [F1] it is shown that if $n\not\in\{ 1,2,3,6\}$ then $n$ is a 
$\theta$-congruent number if and only if the rank of the associated elliptic curve is positive; i.e., 
it is not possible to obtain a corresponding $\theta$-triangle from torsion points on this elliptic curve. 
In [F1] it was shown that the condition $n\not\in\{ 1,2\}$ is indispensible in this statement; i.e.,  
there are torsion solutions for $n=1$ and $n=2$. For $n=3$ and $n=6$ this was left as an open problem, 
which is now answered affirmatively by the above considerations. \end{rem} 

The following result yields a nice geometric characterization of the 4-torsion solutions. 

\begin{thm} 
Let $p,q\in{\N}$ be coprime, let $k\in{\N}$ be squarefree, and let $m=-pk$ and 
$n=qk$. Consider a rational point $P$ on the curve $E(m,n)$, the associated 
$\theta$-congruent triangle $\Delta$ and the associated triplet $T$ of rational 
squares in an arithmetic progression. Then the following statements are equivalent: 
\begin{enumerate}[\upshape (i)]
\item $P$ has order four;
\item $T$ contains the number zero; 
\item $\Delta$ is isosceles, i.e., has two equal sides.
\end{enumerate} 
\end{thm}
\begin{proof}
Under the correspondences in Thm 1.4 we have $a=b$ for the sides of $\Delta$ if 
and only if $\alpha =0$ for the smallest element in the progression $\alpha^2<\beta^2
<\gamma^2$, which is the case if and only if $X_2=0$ for a corresponding point 
$(X_0,X_1,X_2,X_3)$ on $Q(m,n)$ as defined in Thm 2.1 (since $\alpha =X_2/X_1$, $\beta=X_0/X_1$ 
and $\gamma =X_3/X_1$ under the correspondence between points on this curve and rational 
squares in arithmetic progression). Now the condition $X_2=0$ corresponds to the equation  
$(x+m)(y^2+m(x+n)^2)=0$ via the biregular mapping in Thm 2.1. This equation yields either 
$x=-m$ (which is one of the 2-torsion points) or else $y^2+m(x+n)^2=0$. In the latter case 
$-m$ must clearly be a square. Since both $y^2=-m(x+n)^2$ and $y^2=x(x+m)(x+n)$, we have  
$$\begin{array}{ll} 
  0 &= x(x+m)(x+n) + m(x+n)^2 \\
    &= (x+n)(x^2+2mx+mn) \\
    &= (x+n)\bigl( (x+m)^2+m(n-m)\bigr) \end{array}$$ 
and hence $(x+m)^2=-m(n-m)$. Hence not only $-m$, but also $n-m$ must be a square. 
This is exactly the condition that $E(m,n)$ contains points with $4$-torsion. Writing 
$-m=u^2$, the equation $y^2=-m(x+n)^2=u^2(x+n)^2$ yields $y=\pm u(x+n)$, which shows 
that the points $(x,y)$ on $E(m,n)$ associated with isosceles triangles are exactly the 
points of order $4$.  
\end{proof} 

The isosceles triangles occurring as solutions of the generalized congruent number problem 
can be characterized in terms of intrinsic geometrical properties, as we now show.
(Note that when we speak of an isosceles $\theta$-triangle we always assume the angle 
$\theta$ to be between the two equal sides.) 

\begin{thm} There is an isosceles rational $\theta$-triangle if
and only if $\sin(\theta/2)$ is rational. If $\sin(\theta/2)=\varrho/\sigma$ 
where $\varrho,\sigma\in\mathbb{N}$ are coprime, then the sides of
the unique rational $\theta$-triangle with squarefree $k$ are given
by $a,a,c$ with $a=k\sigma$ and $c=2ak\sin(\theta/2)=2k\varrho$.
\end{thm}

\begin{proof} First, if $a,a,c$ are the rational sides of an isosceles
rational $\theta$-triangle with a rational cosine $\cos(\theta)=r/s$, then
$c=2a\sin(\theta/2)$ so that $\sin(\theta/2)$ is also rational. Conversely, 
if $\sin(\theta/2)$ is rational, say $\sin(\theta/2)=\varrho/\sigma$ with 
$\varrho,\sigma\in\mathbb{N}$ coprime, then $\cos(\theta)=1-2\sin^2(\theta/2)$
is rational. If $\sigma$ is odd, then $\cos(\theta)=(\sigma^{2}-2\varrho^{2})/\sigma^{2}=r/s$
is a coprime representation with $r,s$ both odd. Let $a=2\sigma$ and $c=2a\sin(\theta/2)=
4\varrho$. Then the isosceles rational $\theta$-triangle with sides $a,a,c$ has the area 
$(a^2/2)\sin(\theta)=2\sigma^2\sqrt{1-\cos^{2}(\theta)} = 2\sigma^2\sqrt{(s^2-r^2)/s^2} = 
(2\sigma^2/s)\sqrt{s^2-r^2} = 2\sqrt{s^2-r^2}$, so that this triangle is the unique solution 
with $k=2$.\par\indent 
If $\sigma=2\tau$ is even, then $\cos(\theta)=(4\tau^{2}-2\varrho^{2})/(4\tau^{2}) = 
(2\tau^{2}-\varrho^{2})/(2\tau^{2})=r/s$ is a coprime representation with $s$ even and $r$ odd. 
Let $a=\sigma$ and $c=2a\sin(\theta/2)=2\varrho$. Then the isosceles rational $\theta$-triangle
with sides $a,a,c$ has the area $(a^2/2)\sin(\theta) = (\sigma^2/2)\sqrt{1-\cos^{2}(\theta)} = 
(\sigma^2/2)\sqrt{(s^2-r^2)/s^2} = \bigl(\sigma^2/(2s)\bigr)\sqrt{s^2 -r 2} = 
\bigl(\sigma^2/(4\tau^2)\bigr)\sqrt{s^2-r^2} = \sqrt{s^2-r^2}$, so that this triangle 
is the unique solution with $k=1$. $\square$.
\end{proof} 

\begin{rem} The equilateral triangle with all three sides of length $2$ plays a somewhat 
special role. First, it is the unique isosceles rational $\pi/3$-triangle with area $k\sqrt{3}$ 
and squarefree $k\in\mathbb{N}$; in fact, $k=1$. Furthermore, there is no other rational 
$\pi/3$-triangle with area $\sqrt{3}$, since the rank of the Mordell-Weil group 
$E_{\mathbb{Q}}(-1,3)$ is zero (cf. [O1] or [Y1]). For any squarefree $k>1$ there
are either no rational $\pi/3$-triangles with area $k\sqrt{3}$ at all, or else there are 
infinitely many such triangles. The only case in which there is a single solution is the case 
$k=1$ with the above-mentioned equilateral triangle. 
\end{rem} 

The possibility of freely moving between the concordant form problem and the generalized congruent number 
problem provides a way of translating solutions found for one of these problems to solutions of the 
other problem. For example, various interesting examples of concordant forms and of triplets of rational 
squares in arithmetic progressions can be obtained from the examples for rational $2\pi/3$-triangles found 
in [Ka]. Conversely, examples for concordant forms given in [O1] can be used to construct rational 
$\theta$-triangles.

\end{document}